\documentclass{article}

\usepackage{amsmath,amsthm,amsopn,amstext,amscd,amsfonts,amssymb,verbatim}
\usepackage{natbib}

\usepackage{graphicx}
\usepackage{verbatim,color}
\usepackage{caption}
\usepackage{subcaption}
\usepackage{tcolorbox}
\usepackage{mathtools}

\def\tr{\mathop{\rm tr}\nolimits}
\def\etr{\mathop{\rm etr}\nolimits}

\def\vol {\mathop{\rm Vol}\nolimits}
\def\diag{\mathop{\rm diag}\nolimits}

\def\rank{\mathop{\rm rank}\nolimits}
\def\ch{\mathop{\rm ch}_{i}\nolimits}

\newcommand {\boldgreektext}[1] {\boldmath
             \(#1\)\unboldmath}
\newcommand {\boldgreek}[1]
             {\mbox{\boldgreektext{#1}}
            }

\renewenvironment{abstract}
                 {\vspace{6pt}
                  \begin{center}
                  \begin{minipage}{5in}
                  \centerline{\textbf{Abstract}}
                  \noindent\ignorespaces
                 }
                 {\end{minipage}\end{center}}
                 
\newtheorem{theorem}{\textbf{Theorem}}[section]

\newtheorem{proposition}{\textbf{Proposition}}[section]
\theoremstyle{definition}
\newtheorem{definition}{\textbf{Definition}}[section]

\newtheorem{remark}{\textbf{Remark}}[section]

\title{\Large \textbf{About the matrix variate problem involved in the distribution of $\mathbf{E}^{-1}\mathbf{H}$}}
\author{
  \textbf{Jos\'e A. D\'{\i}az-Garc\'{\i}a} \thanks{Corresponding author\newline
    }\\
  {\normalsize Universidad Aut\'onoma de Chihuahua} \\
  {\normalsize Facultad de Zootecnia y Ecolog\'{\i}a} \\
  {\normalsize Perif\'erico Francisco R. Almada Km 1, Zootecnia} \\
  {\normalsize 33820 Chihuahua, Chihuahua, M\'exico}\\
  {\normalsize E-mail: jadiaz@uach.mx}\\
  \textbf{Francisco J. Caro-Lopera}\\
  {\normalsize University of Medellin} \\
  {\normalsize Faculty of Basic Sciences} \\
  {\normalsize Carrera 87 No.30-65} \\
  {\normalsize Medell\'{\i}n, Colombia} \\
  {\normalsize E-mail: fjcaro@udemedellin.edu.co} \\[2ex]
}
\date{}
\begin{document}
\maketitle

\begin{abstract}
This work studies the distribution of the nonsymmetric matrix $\mathbf{E}^{-1}\mathbf{H}$. This random product is of fundamental interest under the general multivariate linear hypothesis setting. Specifically when $\mathbf{H}$ and $\mathbf{E}$ are seen as the sums of squares and the sums of products due to the hypothesis and due to the error, respectively.
\end{abstract}                 

\section{Introduction}\label{sec1}

The beta type II distribution play a fundamental roll with testing equality of two covariance matrices, with the MANOVA case, and with testing independence of two sets of variates. In particular in MANOVA case, any invariant test depends only on latent roots of matrix $\mathbf{E}^{-1}\mathbf{H}$, see \citet[Theorem 10.2.1, p. 437]{mh:05}. At first, it would seem of interest to find the distribution of the matrix $\mathbf{E}^{-1}\mathbf{H}$ and thereby subsequently study the joint distribution of its latent roots. However, since $\mathbf{E}^{-1}\mathbf{H}$, $\mathbf{H}^{1/2}\mathbf{E}^{-1}\mathbf{ H} ^{1/2}$, $\mathbf{E}^{-1/2}\mathbf{H}\mathbf{E}^{-1/2}$ have the same latent roots and the last two are matrices symmetric, it was no longer of interest to study the distribution of the non-symmetric matrix $\mathbf{E}^{-1}\mathbf{H}$. This is due to its complexity and to the fact that, in general, it is more accessible to study the symmetric case due to all the development obtained regarding results in distribution theory, numerical methods and matrix algebra in the case of symmetric matrices. As a consequence, without exception, all authors in the statistical literature addressed the problem from the two symmetric expressions, both in classical and generalised multivariate statistical analysis, see \citet{r:57}, \citet{sk:79}, \citet{a:82},  \citet{fz:90}, \citet{re:02}, \citet{mh:05}, and \citet{gvb:13}, and references therein. We have placed the discussion of this problem in the following stages: Section \ref{sec2} provides two Jacobians  and collects some necessary results for the development of the main result. Then the distribution of the nonsymmetric matrix $\mathbf{E}^{-1}\mathbf{H}$ is obtained in Section \ref{sec3}.

\section{Notation and preliminaries results}\label{sec2}

The discussion is based on the following results.
\subsection{Notation} 
Let $\mathcal{L}_{m,n}$ be the \textit{linear space} of all $n \times m$ matrices of rank $m \leq n$ on $\Re$ (the set of real numbers). The \textit{general linear group} over $\Re$ is the group of $\mathbf{A} \in \mathcal{L}_{m,m}$ invertible matrices, and shall be denoted by $GL(m)$. The subgroup $\mathcal{O}(m)$ of the general linear group $GL(m)$ is defined as $\mathcal{O}(m)=\{\mathbf{G} \in GL(m)|\mathbf{G}'\mathbf{G}=\mathbf{GG}'=\mathbf{I}_{m}\}$, that is, \textit{the group of $m \times m$ orthogonal matrices}; where $\mathbf{G}'$ denotes the \textit{transpose} of $\mathbf{G}$. If $\mathbf{A} \in GL(m)$ its $i-$th  \textit{latent root} $\theta_{i}$ shall be denoted as $\ch(\mathbf{A}) = \theta_{i}$. We denote by ${\mathfrak S}_{m}$ the \textit{real vector space} of all $\mathbf{S} \in \mathcal{L}_{m,m}$ such that $\mathbf{S} = \mathbf{S}'$. Let $\mathfrak{P}_{m}$ be the \emph{cone of positive definite matrices} $\mathbf{S} \in GL(m)$; then $\mathfrak{P}_{m}$ is an open subset of ${\mathfrak S}_{m}$.

\subsection{Jacobians}

Let $\mathbf{C}$ and $\mathbf{X}$ positive definite matrices (hence symmetric matrices) and define $\mathbf{Y} = \mathbf{CX}$, therefore $\mathbf{Y}$ is not necessarily a symmetric matrix. In this section the Jacobian of this transformation is calculated assuming that $\mathbf{Y}$ is a symmetric and a non-symmetric matrix. 

\begin{theorem}\label{teo1}
  Let $\mathbf{Y} \in GL(m)$ and $\mathbf{C}, \mathbf{X} \in \mathfrak{P}_{m}$. Define $\mathbf{Y} = \mathbf{XC}$, then
  
  \begin{equation}\label{eq1}
    (d\mathbf{Y}) = |\mathbf{C}|^{m}\prod_{i<j=1}^{m}(\lambda_{i}+\lambda_{j})(d\mathbf{X})(\mathbf{G}'d\mathbf{G})
  \end{equation}
  where $\lambda_{i}= \ch(\mathbf{X})$, $i = 1, \dots,m$, $\mathbf{G} \in \mathcal{O}(m)$ and $(\mathbf{G}'d\mathbf{G})$ is the Haar measure on $\mathcal{O}(m)$, see \citet[Sectiom 2.1.4, p. 67]{mh:05}.
\end{theorem}
\begin{proof}
  Let $\mathbf{Y} = \mathbf{XC}$ and define $\mathbf{U} = \mathbf{Y}'\mathbf{Y}$ and $\mathbf{V} = \mathbf{X}'\mathbf{X} = \mathbf{X}^{2} = \mathbf{V}'$. Thus
  $$
    \mathbf{U}=\mathbf{Y}'\mathbf{Y}= (\mathbf{XC})'\mathbf{XC} = \mathbf{C}'\mathbf{X}'\mathbf{XC} = \mathbf{C}'\mathbf{V}\mathbf{C} = \mathbf{C}\mathbf{V}\mathbf{C},
  $$
  hence, by \citet[Theorem 2.1.6, p. 58]{mh:05} we have
  \begin{equation}\label{eq2}
    (d\mathbf{U}) = |\mathbf{C}|^{m+1} (d\mathbf{V}).
  \end{equation}
  Now, observe that $\mathbf{U} = \mathbf{Y}'\mathbf{Y}$, then by \citet[Theorem 2.1.14, p. 66]{mh:05}
  $$
    (d\mathbf{Y}) = 2^{-m}|\mathbf{U}|^{-1/2}(d\mathbf{U})(\mathbf{G}'d\mathbf{G}),
  $$
  where $\mathbf{H }\in \mathcal{O}(m)$, thus
  \begin{eqnarray}
    (d\mathbf{U}) &=& 2^{m} |\mathbf{U}|^{1/2}(d\mathbf{Y})(\mathbf{G}'d\mathbf{G})^{-1} \nonumber \\
      &=& 2^{m} |\mathbf{Y}'\mathbf{Y}|^{1/2}(d\mathbf{Y})(\mathbf{G}'d\mathbf{G})^{-1} \nonumber \\
      &=& 2^{m} |\mathbf{Y}|(d\mathbf{Y})(\mathbf{G}'d\mathbf{G})^{-1}. \label{eq3}
  \end{eqnarray}
\begin{scriptsize} 
\begin{remark}
Note that the use of \citet[Theorem 2.1.14, p. 66]{mh:05}, implies the Jacobian
\begin{equation}\label{r11}
  (d\mathbf{Y}) = 2^{-m}|\mathbf{U}|^{-1/2}(d\mathbf{U})(\mathbf{G}'d\mathbf{G}).
\end{equation}
This result is obtained regardless of the assumption of any of the following three factorizations, among other possibilities, see \citet{dggf:05}:
\begin{equation}\label{u}
  \mathbf{Y}=
  \left\{
    \begin{array}{ll}
      \mathbf{GT}, & \hbox{QR decomposition;} \\
      \mathbf{GW}, & \hbox{polar decomposition;} \\
      \mathbf{QDG}', & \hbox{SVD decomposition.}
    \end{array}
  \right.
\end{equation}
Here $\mathbf{G}, \mathbf{Q} \in \mathcal{O}(m)$, $\mathbf{W} \in \mathfrak{P}_{m}$, $\mathbf{T}$ is an upper triangular matrix and $\mathbf{D} = \diag(d_{1}, \dots, d_{m})$. Then after defining $\mathbf{U} = \mathbf{Y}'\mathbf{Y}$,
\begin{equation}\label{yy}
  \mathbf{U} = \mathbf{Y}'\mathbf{Y}=
  \left\{
    \begin{array}{ll}
      \mathbf{T}'\mathbf{T}, & \hbox{Cholescky decomposition;} \\
      \mathbf{W}^{2}, & \hbox{Square root positive definite;} \\
      \mathbf{GD}^{2}\mathbf{G}' = \mathbf{GL}\mathbf{G}', & \hbox{Espectral decomposition,}
    \end{array}
  \right.
\end{equation}
where $\mathbf{L}= \mathbf{D}^{2} =  \diag(l_{1}, \dots, l_{m})$. i.e. $l_{i} = d_{i}^{2}$, $i = 1,\dots,m$.
In terms of the Jacobians for  (\ref{u}) and the corresponding Jacobians for the factorisations (\ref{yy}), then the Jacobian (\ref{r11}) is obtained.
\end{remark}
\end{scriptsize}

Similarly given that, $\mathbf{V} = \mathbf{X}^{2}$, by \citet[Theorem 1.35 and Corollary 2.8.1, pp. 66 and 98, respectively]{mam:97}
\begin{equation}\label{eq4}
  (d\mathbf{V}) = 2^{m} |\mathbf{X}| \prod_{i<j=1}^{m}(\lambda_{i}+\lambda_{j})(d\mathbf{X}),
\end{equation}
where $\lambda_{i}= \ch(\mathbf{X})$, $i = 1, \dots,m$.
Hence, substituting (\ref{eq3}) and (\ref{eq4}) in (\ref{eq2}) we obtain
$$
  2^{m} |\mathbf{Y}|(d\mathbf{Y})(\mathbf{G}'d\mathbf{G})^{-1} = |\mathbf{C}|^{m+1} 2^{m}|\mathbf{X}|\prod_{i<j=1}^{m}(\lambda_{i}+\lambda_{j})(d\mathbf{X}),
$$
  as $\mathbf{Y} = \mathbf{XC}$, hence $|\mathbf{Y}| = |\mathbf{X}||\mathbf{C}|$ from where
  $$
    (d\mathbf{Y}) = |\mathbf{X}|^{-1}|\mathbf{C}|^{-1}|\mathbf{C}|^{m+1}  |\mathbf{X}|\prod_{i<j=1}^{m}(\lambda_{i}+\lambda_{j})(d\mathbf{X})(\mathbf{G}'d\mathbf{G}),
  $$
  and the desired result is achieved.
\end{proof}

The following result assumes that $\mathbf{Y}=\mathbf{Y}'$.

\begin{theorem}\label{teo2}
  Let $\mathbf{Y}, \mathbf{C}, \mathbf{X} \in \mathfrak{P}_{m}$. Define $\mathbf{Y} = \mathbf{XC}$, then  
  \begin{equation}\label{eqc1}
    (d\mathbf{Y}) = |\mathbf{C}|^{m}\prod_{i<j=1}^{m}\left(\frac{\lambda_{i}+\lambda_{j}}{\theta_{i}+\theta_{j}}\right)(d\mathbf{X})
  \end{equation}
  where $\lambda_{i}= \ch(\mathbf{X})$ and $\theta_{i}= \ch(\mathbf{XC})$ , $i = 1, \dots,m$. 
\end{theorem}
\begin{proof}
  Let $\mathbf{Y} = \mathbf{XC}$ and define $\mathbf{U} = \mathbf{Y}'\mathbf{Y} = \mathbf{Y}^{2}$ and $\mathbf{V} = \mathbf{X}'\mathbf{X} = \mathbf{X}^{2} $. Thus
  $$
    \mathbf{U}=\mathbf{Y}'\mathbf{Y}=\mathbf{Y}^{2} = (\mathbf{XC})^{2} = (\mathbf{XC})'\mathbf{XC} = \mathbf{C}\mathbf{X}^{2}\mathbf{C} =\mathbf{C}\mathbf{V}\mathbf{C},
  $$
  hence, by \citet[Theorem 2.1.6, p. 58]{mh:05} we have
  \begin{equation}\label{eqc2}
    (d\mathbf{U}) = |\mathbf{C}|^{m+1} (d\mathbf{V}).
  \end{equation}
  Now, observe that $\mathbf{U} = \mathbf{Y}^{2}$, then by \citet[Theorem 1.35 and Corollary 2.8.1, pp. 66 and 98, respectively]{mam:97}  
\begin{equation}\label{eqc3}
  (d\mathbf{U}) = 2^{m} |\mathbf{Y}| \prod_{i<j=1}^{m}(\theta_{i}+\theta_{j})(d\mathbf{Y}).
\end{equation}
 where $\lambda_{i}= \ch(\mathbf{X})$, $i = 1, \dots,m$.   
 Also, $\mathbf{V} = \mathbf{X}^{2}$, 
  \begin{equation}\label{eqc4}
    (d\mathbf{V}) = 2^{m} |\mathbf{X}| \prod_{i<j=1}^{m}(\lambda_{i}+\lambda_{j})(d\mathbf{X}),
  \end{equation}
  where $\lambda_{i}= \mathrm{ch}_{i}(\mathbf{X})$, $i = 1, \dots,m$.
  Therefore, substituting (\ref{eqc3}) and (\ref{eqc4}) in (\ref{eqc2}) we get
  $$
    2^{m} |\mathbf{Y}| \prod_{i<j=1}^{m}(\theta_{i}+\theta_{j})(d\mathbf{Y}) = |\mathbf{C}|^{m+1} 2^{m} |\mathbf{X}|\prod_{i<j=1}^{m}(\lambda_{i}+\lambda_{j})(d\mathbf{X}),
  $$
  as $\mathbf{Y} = \mathbf{XC}$. Thus $|\mathbf{Y}| = |\mathbf{X}||\mathbf{C}|$ and 
  $$
    (d\mathbf{Y}) = |\mathbf{X}\mathbf{C}|^{-1}|\mathbf{C}|^{m+1}  |\mathbf{X}| \prod_{i<j=1}^{m} \left(\frac{\lambda_{i} +\lambda_{j}}{\theta_{i} +\theta_{j}}\right)(d\mathbf{X}),
  $$
  and the result is derived.
\end{proof}

\subsection{Matrix variate beta type II distribution}

Assume that $\mathbf{Y} =(\mathbf{Y}'_{1} \mathbf{Y}'_{2})'$ has a matrix variate normal distribution such that $\mathbf{Y}_{1} \in \Re^{a \times m}$ and $\mathbf{Y}_{2} \in \Re^{b \times m}$, that is, $\mathbf{Y} \in \Re^{a+b \times m}$, see \citet[Theorem 3.1.1, p.79]{mh:05}. Hence, the  matrix beta type II is defined as 
\begin{equation}\label{defbII}
 \mathbf{ F} =
  \left\{%
     \begin{array}{ll}
      (\mathbf{Y}'_{2}\mathbf{Y}_{2})^{-1/2}(\mathbf{Y}'_{1}\mathbf{Y}_{1}) ((\mathbf{Y}'_{2}\mathbf{Y}_{2})^{-1/2})', & \mbox{Definition 1},\\
      (\mathbf{Y}'_{1}\mathbf{Y}_{1})^{1/2}(\mathbf{Y}'_{2}\mathbf{Y}_{2})^{-1} ((\mathbf{Y}'_{1}\mathbf{Y}_{1}) ^{1/2})', & \mbox{Definition 2},\\
      \mathbf{Y}_{1}(\mathbf{Y}'_{2}\mathbf{Y}_{2})^{-1} \mathbf{Y}'_{1}, & \mbox{Definition 3}.\\
    \end{array}%
  \right.
\end{equation}
where $\mathbf{C}^{1/2}(\mathbf{C}^{1/2})' = \mathbf{C}$ is a reasonable nonsingular factorization of $\mathbf{C}$ (non-negative definite square root, Cholesky factorisation, etc.), see \citet{sk:79}, \citet{mh:05}, \citet{fz:90} and \citet{gvb:13}. Observe that the random matrices $\mathbf{H} = \mathbf{Y}'_{1}\mathbf{Y}_{1}$ and $\mathbf{E} = \mathbf{Y}'_{2}\mathbf{Y}_{2}$ have a central Whishart distribution under the null hypothesis, see \citet[Section 3.2, p. 85]{mh:05}. Moreover, in the setting of definitions
1 and 2 its density function is 

\begin{equation}\label{beta}
 dF_{\mathbf{F}}(\mathbf{F}) = \frac{1}{\mathcal{B}_{m}[a/2,b/2]} |\mathbf{F}|^{(a -m -1)/2} |\mathbf{I}_{m} + \mathbf{F}|^{-(a+b)/2} (d\mathbf{F}),
     \quad \mathbf{F} \in \mathfrak{P}_{m},
\end{equation}
with $a \geq m$ and $b \geq m$; where
$\mathcal{B}_{m}[a,b]$ denotes the multivariate beta function defined by
$$
  \mathcal{B}_{m}[r,q] =  \int_{\mathbf{R} \in \mathfrak{P}_{m}} |\mathbf{R}|^{r-(m+1)/2} |I_{m} + \mathbf{R}|^{-(r+q)} (d\mathbf{R}) = \frac{\Gamma_{m}[r]
  \Gamma_{m}[q]}{\Gamma_{m}[r+q]},
$$
here Re$(r) > (m-1)/2$, Re$(q)> (m-1)/2$ and  $\Gamma_{m}[r]$ denotes the multivariate gamma function which is defined by
$$
  \Gamma_{m}[r] = \int_{\mathbf{R} \in \mathfrak{P}_{m}} \etr(-\mathbf{R}) |\mathbf{R}|^{r-(m+1)/2} (d\mathbf{R}),
$$
and $\etr(\cdot) \equiv \exp(\tr(\cdot))$, see \citet{h:55}.

Note that the density, properties and associated distributions under Definition 3 (for $a < m$) are obtained from the
density and properties via definitions 1 or 2 ($a>m$) by the following substitutions:
\begin{center}
\begin{equation}\label{sust}
    m \rightarrow a, \quad a \rightarrow m, \quad b \rightarrow b+a-m,
\end{equation}
\end{center}
see \citet[p. 96]{sk:79} or \citet[eq. (7), p. 455]{mh:05}.

Finally, observe that the matrix variate beta type II distribution is invariant under the  family of elliptically contoured distributions, i. e. the matrix variate normal distribution assumption of $\mathbf{Y} =(\mathbf{Y}'_{1} \mathbf{Y}'_{2})'$ can be replaced by considering $\mathbf{Y} =(\mathbf{Y}'_{1} \mathbf{Y}'_{2})'$ with a  matrix variate elliptically contoured distribution. See \citet[Theorem 5.12, p. 139 and Theorem 5.15, p. 142]{gvb:13} for a theory of matrix variate vector-elliptically contoured distributions and check \citet[Section 3.5.4, p. 114]{fz:90} for a general matrix variate elliptically contoured distribution setting.

\section{Main result}\label{sec3}

This section attains the distribution of a  \textit{nonsymmetric matrix} $\mathbf{F}_{1} = \mathbf{E}^{-1}\mathbf{H}$, when $\mathbf{E}$ and $\mathbf{H}$ have a central Wishart distribution. Note that the matrix $\mathbf{F}_{1}$ is the nonsymmetric version of the matrix beta type II.

The matrix variate beta type II distribution is crucial in \textbf{general multivariate linear hypothesis}. Specifically, the random matrices $\mathbf{H}$ and $\mathbf{E}$ define the matrices of sums of squares and sums of products due to the hypothesis and due to the error, respectively, see \citet[Section 10.2, p. 432]{mh:05}. In addition, although there is no uniformly most powerful invariant test and many functions of the latent roots of $\mathbf{F}$ (or $\mathbf{F}_{1}$) have been proposed as test statistics, any invariant test depends only on $l_{1}, \cdots,l_{s}$, where $s = \rank (\mathbf{E}^{-1} \mathbf{H})$, see \citet[Theorem 10,2,1, p. 437]{mh:05}.

As a consequence, the problem of finding the distribution of $\mathbf{F}_{1} = \mathbf{E}^{-1}\mathbf{H}$ seems to be out of an historical interest because the matrices defined in (\ref{defbII}) and $\mathbf{F}_{1}$ have the same latent roots, i. e. $\ch(\mathbf{F}_{1}) = \ch(\mathbf{E}^{-1}\mathbf{H}) = \ch(\mathbf{H}\mathbf{E}^{-1}) = \ch(\mathbf{E}^{-1/2}\mathbf{H}\mathbf{E}^{-1/2})= \ch(\mathbf{H}^{1/2}\mathbf{E}^{-1}\mathbf{H}^{1/2}) = \ch(\mathbf{Y}_{1}\mathbf{E}^{-1}\mathbf{Y}'_{1}) = \ch(\mathbf{F})$.

However, the mathematical related problem behind the density function of $\mathbf{F}_{1} = \mathbf{E}^{-1}\mathbf{H}$ reveals a number of complex stages that we try to explore next.

\begin{theorem} \label{teo3}
  Let $\mathbf{H}$ and $\mathbf{E}$ be independent random matrices such that $\mathbf{H}$ has a $m$-dimensional Wishart distribution with $a \geq m$ degrees of freedom and expectation $a\mathbf{I}_{m}$, this is  $\mathbf{H} \sim \mathcal{W}_{m}(a,\mathbf{I}_{m})$ and $\mathbf{E} \sim \mathcal{W}_{m}(b,\mathbf{I}_{m})$, $b \geq m$. Then the density function of $\mathbf{F}_{1} = \mathbf{E}^{-1}\mathbf{H}$ is
\begin{equation}\label{F11}
  dF_{\mathbf{F}_{1}}(\mathbf{F}_{1}) =\frac{|\mathbf{F}_{1}|^{(a+m-1)/2}}{2^{(a+b)m/2}\Gamma_{m}[a/2]\Gamma_{m}[b/2]} \; J_{1}\quad(d\mathbf{F}_{1})
\end{equation}
where $J_{1}$ is the double integral
\begin{equation}\label{J1}
   \int_{\mathbf{U} \in \mathfrak{P}_{m}} |\mathbf{U}|^{(a+b)/2-(m+1)} \frac{\etr[-\frac{1}{2}(\mathbf{I}_{m}+\mathbf{F}_{1})\mathbf{U}]} 
   {\displaystyle \prod_{i<j=1}^{m}(\lambda_{i}+\lambda_{j})}(d\mathbf{U})\int_{\mathbf{G} \in \mathcal{O}(m)}(\mathbf{G}'d\mathbf{G})^{-1}
\end{equation}
and $\lambda_{i} = \ch(\mathbf{UF})$ and  $(\mathbf{G}'d\mathbf{G})$ is the Haar measure on $\mathcal{O}(m)$, see \citet[Sectiom 2.1.4, p. 67]{mh:05}.
\end{theorem}
\begin{proof}
The join density of $\mathbf{H}$ and $\mathbf{E}$ is
$$
  \frac{|\mathbf{H}|^{(a-m-1)/2}|\mathbf{E}|^{(b-m-1)/2}}{2^{(a+b)m/2}\Gamma_{m}[a/2]\Gamma_{m}[b/2]}\etr\left[-\frac{1}{2}(\mathbf{H}+\mathbf{E})\right] (d\mathbf{H})\wedge(d\mathbf{E}).
$$
Making the change of variables $\mathbf{F}_{1} = \mathbf{E}^{-1}\mathbf{H}$ and $\mathbf{U} = \mathbf{E}$, then, $\mathbf{H} = \mathbf{EF}_{1}= \mathbf{UF}_{1}$ and $\mathbf{E} = \mathbf{U}$, and by Theorem \ref{teo1}
$$
  (d\mathbf{H})\wedge(d\mathbf{E}) = |\mathbf{F}_{1}|^{m} \frac{(d\mathbf{U})\wedge(d\mathbf{F}_{1})}{\displaystyle \prod_{i<j=1}^{m}(\lambda_{i}+\lambda_{j})(\mathbf{G}'d\mathbf{G})},
$$
with $\lambda_{i} = \ch(\mathbf{H} = \ch(\mathbf{UF}_{1})$, $i = 1,\cdots, m$.
Therefore the joint density of $\mathbf{F}_{1}$, $\mathbf{G}$ and $\mathbf{U}$ is
$$
  \frac{|\mathbf{F}_{1}|^{(a+m-1)/2}|\mathbf{U}|^{(a+b)/2-(m+1)}}{2^{(a+b)m/2}\Gamma_{m}[a/2]\Gamma_{m}[b/2]}
   \etr\left[-\frac{1}{2}(\mathbf{I}_{m}+\mathbf{F}_{1})\mathbf{U}\right] \frac{(d\mathbf{U})\wedge(d\mathbf{F}_{1})}{\displaystyle \prod_{i<j=1}^{m}(\lambda_{i}+\lambda_{j})(\mathbf{G}'d\mathbf{G})}.
$$
The desired result is obtained by integration over $\mathbf{U} \in \mathfrak{P}_{m}$ and  $\mathbf{G} \in \mathcal{O}(m)$.
\end{proof}
\begin{scriptsize}
\begin{remark}
An alternative proof for Theorem \ref{teo3}, conducts the change of variables $\mathbf{F}_{1} = \mathbf{E}^{-1}\mathbf{H}$ and $\mathbf{V} = \mathbf{H}$, into $\mathbf{E} = (\mathbf{F}_{1}\mathbf{H}^{-1})^{-1}= \mathbf{HF}_{1}^{-1}$ and $\mathbf{H} = \mathbf{V}$. Then by Theorem \ref{teo1},
\begin{eqnarray*}
  (d\mathbf{H})\wedge(d\mathbf{E}) &=& |\mathbf{V}|^{m} \frac{(d\mathbf{V})\wedge(d\mathbf{F}_{1}^{-1})}{\displaystyle \prod_{i<j=1}^{m} (\theta_{i}+\theta_{j})(\mathbf{G}'d\mathbf{G})}, \\
    &=& |\mathbf{V}|^{m} |\mathbf{F}_{1}|^{-(m+1)}\frac{(d\mathbf{V})\wedge(d\mathbf{F}_{1})}{\displaystyle \prod_{i<j=1}^{m}(\theta_{i}+\theta_{j})(\mathbf{G}'d\mathbf{G})}, 
\end{eqnarray*}
where $\theta_{i} = \ch(\mathbf{E}) = \ch(\mathbf{VF}^{-1}_{1})$, $i = 1,\cdots, m$. Thus the marginal density of $\mathbf{F}_{1}$ is
\begin{equation}\label{F12}
  dF_{\mathbf{F}_{1}}(\mathbf{F}_{1}) =\frac{|\mathbf{F}_{1}|^{-(b+m+1)/2}}{2^{(a+b)m/2}\Gamma_{m}[a/2]\Gamma_{m}[b/2]}\; J_{2} \quad (d\mathbf{F}_{1})
\end{equation}
where $J_{2}$ is the double integral
\begin{equation}\label{J2}
  \int_{\mathbf{V}\in\mathfrak{P}_{m}} |\mathbf{V}|^{(a+b)/2-1} 
  \frac{\etr\left[-\frac{1}{2}(\mathbf{I}_{m}+\mathbf{F}_{1}^{-1})\mathbf{V}\right]} {\displaystyle \prod_{i<j=1}^{m}(\theta_{i}+\theta_{j})} (d\mathbf{V})\int_{\mathbf{G}\in\mathcal{O}(m)}(\mathbf{G}'d\mathbf{G})^{-1} .
\end{equation}
\end{remark}
\end{scriptsize}
In both cases, the solutions claims for non available results, hidden underlying in the integrals on $\mathbf{U} \in \mathfrak{P}_{m}$ and  $\mathbf{G} \in \mathcal{O}(m)$ or over $\mathbf{V} \in \mathfrak{P}_{m}$ and on $\mathbf{G} \in \mathcal{O}(m)$. However, an explicit expression for the density of $\mathbf{F}_{1}$ can be obtained by proceeding indirectly, as we proceed next.

We have addressed that matrices $\mathbf{F}_{1}$ and $\mathbf{F}$ involve the same latent roots, and thus attain the same distribution, $dF_{l_{1}, \dots,l_{m}}(l_{1}, \dots,l_{m})$. Let $l_{1}, \dots,l_{m}$, $l_{1} > \cdots >l_{m} > 0$ the latent roots of the random matrix $\mathbf{F}$ (or $\mathbf{F}_{1}$); then, their joint density function is \citep[Corollary 10.4.3, p. 451]{mh:05},
\begin{equation}\label{chs}
  \frac{\pi^{m^{2}/2}}{\Gamma_{m}[m/2]  \mathcal{B}_{m}[a/2,b/2]}\prod_{i=1}^{m}\frac{l_{i}^{(a-m-1)/2}}{(1+l_{i})^{(a+b)/2}} 
  \prod_{i<j=1}^{m}(l_{i}-l_{j}) \bigwedge_{i=1}^{m} dl_{i}.
\end{equation}
By \citet[Theorem 3.3.17, p. 104]{mh:05} (also see \citet{j:64}) with $dF_{\mathbf{F}}(\mathbf{F})$ and $\mathbf{F} = \mathbf{GLG}'$, $\mathbf{G} \in \mathcal{O}(m)$ and $\mathbf{L} = \diag(l_{1},\dots,l_{m})$, $l_{1} > \cdots >l_{m} > 0$,
$$
  dF_{l_{1},\dots,l_{m}}(l_{1},\dots,l_{m}) = \frac{\pi^{m^{2}/2}}{\Gamma_{m}[m/2]} \prod_{i<j=1}^{m}(l_{i}-l_{j})
  \hspace{4cm}
$$
\begin{equation}\label{chs1}\hspace{4cm}
    \times  \int_{\mathbf{G} \in \mathcal{O}(m)}dF_{\mathbf{F}}(\mathbf{GLG}') (d\mathbf{G}) \bigwedge_{i=1}^{m} dl_{i},
\end{equation}
where $(d\mathbf{G})$ is the \textit{normalised invariant measure of Haar} or \textit{Haar invariant distribution on $\mathcal{O}(m)$}, see \citet[Section 2.1.4, pp. 67-72]{mh:05}. Moreover, given that
$$
  \vol[\mathcal{O}(m)] = \int_{\mathbf{G}\in \mathcal{O}(m)}(\mathbf{G}'d\mathbf{G}) = \frac{2^{m} \pi^{m^{2}/2}}{\Gamma_{m}[m/2]},
$$
we have that
$$
  (d\mathbf{G}) = \frac{1}{\vol[\mathcal{O}(m)]}(\mathbf{G}'d\mathbf{G}) = \frac{\Gamma_{m}[m/2]}{2^{m} \pi^{m^{2}/2}}(\mathbf{G}'d\mathbf{G}).
$$
Equalising (\ref{chs1}) and (\ref{chs}) we obtain
$$
  \int_{\mathbf{G} \in \mathcal{O}(m)}dF_{\mathbf{F}}(\mathbf{GLG}') (d\mathbf{G})\bigwedge_{i=1}^{m} dl_{i}= \frac{1}{\mathcal{B}_{m}[a/2,b/2]} \prod_{i=1}^{m}\frac{l_{i}^{(a-m-1)/2}}{(1+l_{i})^{(a+b)/2}}\bigwedge_{i=1}^{m} dl_{i}.
$$
Then, given that $\mathbf{G} \in \mathcal{O}(m)$
\begin{eqnarray*}
  \int_{\mathbf{G} \in \mathcal{O}(m)}\hspace{-0.5cm}dF_{\mathbf{F}}(\mathbf{GLG}') (d\mathbf{G})\wedge (d\mathbf{L}) &=& \frac{1}{\mathcal{B}_{m}[a/2,b/2]} \frac{|\mathbf{L}|^{(a-m-1)/2}}{|\mathbf{I}_{m}+\mathbf{L}|^{(a+b)/2}} (d\mathbf{L}), \\ \hspace{-0.5cm}
   &=& \frac{1}{\mathcal{B}_{m}[a/2,b/2]} \frac{|\mathbf{GLG}'|^{(a-m-1)/2}}{|\mathbf{I}_{m}+\mathbf{GLG}'|^{(a+b)/2}}(d\mathbf{L}). 
\end{eqnarray*}
Thus ($\mathbf{F} = \mathbf{GLG}'$)
$$
  dF_{\mathbf{F}}(\mathbf{F}) \int_{\mathbf{G} \in \mathcal{O}(m)}(d\mathbf{G}) = \frac{1}{\mathcal{B}_{m}[a/2,b/2]} \frac{|\mathbf{F}|^{(a-m-1)/2}}{|\mathbf{I}_{m}+\mathbf{F}|^{(a+b)/2}} (d\mathbf{F}).
$$
Finally
$$
  dF_{\mathbf{F}}(\mathbf{F})= \frac{1}{\mathcal{B}_{m}[a/2,b/2]} \frac{|\mathbf{F}|^{(a-m-1)/2}}{|\mathbf{I}_{m}+\mathbf{F}|^{(a+b)/2}} (d\mathbf{F}).
$$
Now, we apply this approach to the nonsymmetric random matrix $\mathbf{F}_{1}$, which must be a nondefective matrix, see \citet[p. 316]{gvl:96} and \citet[p. 141]{d:97}. For getting this end, consider the following definition, see \citet[Colollary 7.1.8, p. 316]{gvl:96}, \citet[p. 142]{d:97} and \citet[Theorem 15, p. 19]{mn:07}:
\begin{definition}[Jordan decomposition]\label{def1}
Let $\mathbf{A} \in GL(m)$ with distinct latent roots. Then there exist a nonsingular $\mathbf{P} \in GL(m)$ and a diagonal matrix $\boldgreek{\Lambda} \in GL(m)$ whose diagonal elements are the latent roots of $\mathbf{A}$, such that
$$
  \mathbf{PAP}^{-1} = \boldgreek{\Lambda}.
$$
\end{definition}
Now, from \citet[Theorem 2.8, p. 95]{mam:97}, 
\begin{proposition}\label{prop1}
  Let $\mathbf{A} \in GL(m)$ with real distinct positive latent roots $\lambda_{1}>\cdots >\lambda_{m}>0$. Let  $\mathbf{P} \in GL(m)$  such that $\mathbf{A} = \mathbf{P}\boldgreek{\Lambda}\mathbf{P}^{-1}$,  $\boldgreek{\Lambda} = \diag(\lambda_{1},\dots,\lambda_{m})$. Then ignoring the sign,
$$
  (d\mathbf{A})= \prod_{i<j=1}^{m}(\lambda_{i}-\lambda_{j})^{2}(d\boldgreek{\Lambda})\wedge (\mathbf{P}^{-1}d\mathbf{P}),
$$
where $(\mathbf{P}^{-1}d\mathbf{P})$ denotes the wedge product in $\mathbf{P}^{-1}d\mathbf{P}$.
\end{proposition}
Under Proposition \ref{prop1} a generalisation of (\ref{chs1}) for the nonsymmetric case $\mathbf{F}_{1}$  is given by
$$
  dF_{l_{1},\dots,l_{m}}(l_{1},\dots,l_{m}) = \vol[\mathfrak{I}(m)] \prod_{i<j=1}^{m}(l_{i}-l_{j})^{2} 
  \hspace{4cm}
$$
\begin{equation}\label{chs2}\hspace{5cm}
   \int_{\mathbf{P} \in \mathfrak{I}(m)}dF_{\mathbf{F}_{1}}(\mathbf{PLP}^{-1}) (d\mathbf{P}) \bigwedge_{i=1}^{m} dl_{i},
\end{equation} 
where $\mathbf{L} = \diag(l_{1},\dots,l_{m})$, $l_{1} > \cdots >l_{m} > 0$. Thus, if $\mathbf{P} =(\mathbf{p}_{1},\dots,\mathbf{p}_{m})$,  and $\mathbf{P}^{-1} =(\mathbf{q}_{1},\dots,\mathbf{q}_{m})'$, $\mathbf{p}_{i}$, $\mathbf{q}_{i} \in \mathcal{L}_{m,1}$, $i=1,\dots,m$,
$$ 
  \mathfrak{I}(m) =\{\mathbf{P}\in GL(m)|\mathbf{F}_{1}\mathbf{p}_{i} =l_{i}\mathbf{p}_{i} \mbox{ and } \mathbf{q}'_{i} \mathbf{F}_{1}= l_{i} \mathbf{q}'_{i}, i=1,\dots,m\},
$$ 
this is, $\mathbf{p}_{i}$ and $\mathbf{q}_{i}$ are right and left latent vectors for $l_{i}$, $i=1,\dots,m$, of the random matrix $\mathbf{F}_{1}$, see \citet[Propositio 4.2, p. 142]{d:97}. In adition,  
$$
  (d\mathbf{P}) = \frac{1}{\vol[\mathfrak{I}(m)]}(\mathbf{P}^{-1}d\mathbf{P}) \quad\mbox{and}\quad \vol[\mathfrak{I}(m)] = \int_{\mathbf{P} \in \mathfrak{I}(m)}(\mathbf{P}^{-1}d\mathbf{P}),
$$
noting that $\mathbf{P}^{-1}d\mathbf{P} = -d\mathbf{P}^{-1}\mathbf{P}$.
\begin{scriptsize}
\begin{remark}
Observe that in the Jordan decomposition, $\mathbf{A}$, $\mathbf{P} \in GL(m)$ and $\boldgreek{\Lambda} = \diag(\lambda_{1},\dots,\lambda_{m})$  such that $\mathbf{A} = \mathbf{P}\boldgreek{\Lambda}\mathbf{P}^{-1}$. Then $\mathbf{A}$ contains $m^{2}$ functionally independent real variables and $\boldgreek{\Lambda}$ has $m$ functionally independent real variables, then $\mathbf{P}$ involves $m(m-1)$ functionally independent real variables, such that the \textit{parameter count} is, see \citet{er:05},
$$
  m^{2} = m(m-1)+m.
$$
Moreover, the elements of $\mathbf{P} \in GL(m)$ can be regarded as the coordinates of a point on a $m(m - 1)$-dimensional manifold in $GL(m)$, see \citet[Propositio 4.2, p. 142]{d:97}.
\end{remark}
\end{scriptsize}
Proceeding as in the symmetric case, after equalising (\ref{chs2}) and (\ref{chs}) we have
$$
   \int_{\mathbf{P} \in \mathfrak{I}(m)}dF_{\mathbf{F}_{1}}(\mathbf{PLP}^{-1}) (d\mathbf{P}) \bigwedge_{i=1}^{m} dl_{i} = \frac{\pi^{m^{2}/2}}{\Gamma_{m}[m/2]  \mathcal{B}_{m}[a/2,b/2]\vol[\mathfrak{I}(m)]}
\hspace{2cm}
$$
$$\hspace{6cm}
  \times \prod_{i=1}^{m}\frac{l_{i}^{(a-m-1)/2}}{(1+l_{i})^{(a+b)/2}}  \prod_{i<j=1}^{m}(l_{i}-l_{j})^{-1}
  \bigwedge_{i=1}^{m} dl_{i}.
$$
Therefore
$$
   \int_{\mathbf{P} \in \mathfrak{I}(m)}dF_{\mathbf{F}_{1}}(\mathbf{PLP}^{-1}) (d\mathbf{P}) (d\mathbf{L}) = \frac{\pi^{m^{2}/2}}{\Gamma_{m}[m/2]  \mathcal{B}_{m}[a/2,b/2]\vol[\mathfrak{I}(m)]}
\hspace{2cm}
$$
$$\hspace{6cm}
  \times \frac{|\mathbf{L}|^{(a-m-1)/2}}{|\mathbf{I}_{m}+\mathbf{L}|^{(a+b)/2}}  \prod_{i<j=1}^{m}(l_{i}-l_{j})^{-1}
  (d\mathbf{L}).
$$
Then
$$
   \int_{\mathbf{P} \in \mathfrak{I}(m)}dF_{\mathbf{F}_{1}}(\mathbf{PLP}^{-1}) (d\mathbf{P}) (d\mathbf{L}) = \frac{\pi^{m^{2}/2}}{\Gamma_{m}[m/2]  \mathcal{B}_{m}[a/2,b/2]\vol[\mathfrak{I}(m)]}
\hspace{2cm}
$$
$$\hspace{5cm}
  \times \frac{|\mathbf{PLP}^{-1}|^{(a-m-1)/2}}{|\mathbf{I}_{m}+\mathbf{PLP}^{-1}|^{(a+b)/2}}  \prod_{i<j=1}^{m}(l_{i}-l_{j})^{-1}
  (d\mathbf{L}).
$$
Thus 
$$
   dF_{\mathbf{F}_{1}}(\mathbf{F}_{1}) \int_{\mathbf{P} \in \mathfrak{I}(m)}(d\mathbf{P})  = \frac{\pi^{m^{2}/2}}{\Gamma_{m}[m/2]  \mathcal{B}_{m}[a/2,b/2]\vol[\mathfrak{I}(m)]}
\hspace{2cm}
$$
$$\hspace{6cm}
  \times \frac{|\mathbf{F}_{1}|^{(a-m-1)/2}}{|\mathbf{I}_{m}+\mathbf{F}_{1}|^{(a+b)/2}}  \prod_{i<j=1}^{m}(l_{i}-l_{j})^{-1}
  (d\mathbf{F}_{1}).
$$
Finally,
$$
   dF_{\mathbf{F}_{1}}(\mathbf{F}_{1}) = \frac{\pi^{m^{2}/2}}{\Gamma_{m}[m/2]  \mathcal{B}_{m}[a/2,b/2]\vol[\mathfrak{I}(m)]}
\hspace{4cm}
$$
\begin{equation}\label{F1}
  \hspace{4cm}
  \times \frac{|\mathbf{F}_{1}|^{(a-m-1)/2}}{|\mathbf{I}_{m}+\mathbf{F}_{1}|^{(a+b)/2}}  \prod_{i<j=1}^{m}(l_{i}-l_{j})^{-1}
  (d\mathbf{F}_{1}).
\end{equation}
where $l_{i} = \ch(\mathbf{F}_{1})$, $l_{1} > \cdots >l_{m} > 0$.
\begin{scriptsize}
\begin{remark}
Instead of the Jordan decomposition, one would use the \textit{Schur decomposition}, see \citet[Theorem 7.1.3, p. 313]{gvl:96}, \citet[Theorem A9.1, p. 587]{mh:05} and \citet[Theorem 12, p. 17]{mn:07}. Then, considering the Jacobian in \citet[Theorem 2.19, p, 130]{mam:97}, where $\mathbf{T}$ is a lower (or upper) triangular matrix, then, it can be written as $\mathbf{T} = \boldgreek{\Lambda} + \mathbf{M}$ with $\boldgreek{\Lambda} = \diag(\lambda_{1}, \dots,\lambda_{m})$, $\lambda_{1}> \cdots > \lambda_{m} > 0$. Thus $\mathbf{M}$ is strictly lower (upper) triangular matrix ($m_{ii} = 0$, $i = 1, \dots,m$), and a similar expression to (\ref{chs2}) is obtained in terms of the Schur decomposition. In that case, it shall be a function of a double integral, one over the group of orthogonal matrices $\mathcal{O}(m)$ and the other integral runs over the strict triangular matrix $\mathbf{M}$.
\end{remark}
\end{scriptsize}

We end this work by and indirect solution of the the double integrals (\ref{J1}) and (\ref{J2}).

First, equating expressions (\ref{F11}) and (\ref{F1}) gives that
$$
  J_{1} = \frac{2^{(a+b)m/2} \pi^{m^{2}/2}\Gamma_{m}[a/2,b/2]}{\Gamma_{m}[m/2] \vol[\mathfrak{I}(m)]}
            \frac{|\mathbf{F}_{1}|^{-m}}{|\mathbf{I}_{m}+\mathbf{F}_{1}|^{(a+b)/2}} \prod_{i<j=1}^{m}(l_{i}-l_{j})^{-1}
$$
Similarly,  (\ref{F12}) and (\ref{F1}) join with $|\mathbf{I}_{m}+\mathbf{F}_{1}| = |\mathbf{I}_{m}+\mathbf{F}^{-1}_{1}||\mathbf{F}_{1}|$ and for $\mu_{i} = \ch(\mathbf{F}_{1}^{-1})$, then $\mu_{i} = l_{i}^{-1}$, $i = 1,\dots,m$. Thus we have that
$$
  J_{2} = \frac{2^{(a+b)m/2} \pi^{m^{2}/2}\Gamma_{m}[a/2,b/2]}{\Gamma_{m}[m/2] \vol[\mathfrak{I}(m)]}
          |\mathbf{I}_{m}+\mathbf{F}_{1}^{-1}|^{-(a+b)/2} \prod_{i>j=1}^{m}\frac{\mu_{i}\mu_{j}}{(\mu_{j}-\mu_{i})},
$$
where $\mu_{m} > \cdots > \mu_{1}>0$.

The indirect proofs have led some interesting open problems to be clarified in future, in particular Haar inverse measures and $\vol[\mathfrak{I}(m)]$. 



 \end{document}